\documentclass[reqno]{amsart}
\usepackage[scale=0.75, centering, headheight=14pt]{geometry}
\usepackage[latin1]{inputenc}
\usepackage[T1]{fontenc}
\usepackage{lmodern}
\usepackage[english]{babel}
\usepackage{amsmath}

\usepackage{amsmath,amssymb,amsfonts,amsthm}
\usepackage{mathtools,accents}
\usepackage{mathrsfs}
\usepackage{xfrac}
\usepackage{array} 
\usepackage{aliascnt}
\usepackage{booktabs} %%%%% for much better looking tables
\usepackage{array} %%%%% for better arrays (eg matrices) in maths

\usepackage{verbatim} %%%%% adds environment for commenting out blocks of text & for better verbatim
\usepackage{subfig} %%%%% make it possible to include more than one captioned figure/table in a single float

\usepackage{mathrsfs, dsfont}
\usepackage{amssymb}
\usepackage{amsthm}
\usepackage{amsmath,amsfonts,amssymb,esint}
\usepackage{graphics,color}
\usepackage{enumerate}
\usepackage{mathtools,centernot}
\usepackage{esint}

\usepackage{microtype}
\usepackage{paralist} %%%%% very flexible & customisable lists (eg. enumerate/itemize, etc.)
\usepackage{cases}
\usepackage[initials, alphabetic]{amsrefs}
\allowdisplaybreaks

\usepackage{braket}
\usepackage{bm}

\usepackage[citecolor=blue,colorlinks]{hyperref}
\addto\extrasenglish{}

\usepackage{enumerate}
\usepackage{xcolor}
\usepackage{aliascnt}

\makeatletter
\def\newaliasedtheorem#1[#2]#3{
  \newaliascnt{#1@alt}{#2}
  \newtheorem{#1}[#1@alt]{#3}
  \expandafter\newcommand\csname #1@altname\endcsname{#3}
}
\makeatother

\usepackage{indentfirst}

\usepackage{esint}
\usepackage{mathrsfs}
\usepackage{stmaryrd}

\makeatletter
\newsavebox{\measure@tikzpicture}

\newcommand{\setword}[2]{%
  \phantomsection
  #1\def\@currentlabel{\unexpanded{#1}}\label{#2}%
}

\renewcommand\labelenumi{(\roman{enumi})}
\renewcommand\theenumi\labelenumi

\newtheorem{theorem}{\bf Theorem}[section]

\newtheorem{proposition}[theorem]{\bf Proposition}

\newtheorem{remark}[theorem]{\bf{Remark}}
\newtheorem{definition}[theorem]{\bf Definition}
\newtheorem{lemma}[theorem]{\bf Lemma}

\newtheorem*{theorem*}{Theorem}

\newcommand{\eps}{\varepsilon}
\newcommand{\R}{\mathbb R}
\newcommand{\N}{\mathbb N}

\newcommand{\diam}{\operatorname{diam}}

\newcommand{\divergence}{\operatorname{div}}

\newcommand{\sym}{\operatorname{sym}}

\newcommand{\supp}{\operatorname{supp}}

\newcommand{\Tr}{\operatorname{Tr}}
\newcommand{\Leb}{\mathcal{L}}

\newcommand{\MF}{\mathcal{M}}

\DeclareMathOperator{\loc}{loc}

\DeclareMathOperator{\app}{app}

\DeclareMathOperator{\BMO}{BMO}

\newcommand{\setR}{\mathbb{R}}

\newcommand{\abs}[1]{\left\lvert#1\right\rvert}
\newcommand{\norm}[1]{\left\lVert#1\right\rVert}

\newcommand{\di}{\mathop{}\!\mathrm{d}}

\let\div\undefined
\DeclareMathOperator{\div}{div}

\allowdisplaybreaks
\numberwithin{equation}{section}

\title[Regularity estimates in transport equations]{Regularity estimates in transport equations via heat flow and quantitative differentiation}

\author[Elia Bru\'e]{Elia Bru\'e}
\address{Bocconi University, Department of Decision Sciences, Via Serfatti 25, 20136 Milano, Italy}
\email{elia.brue@unibocconi.it}

\author[Maria Colombo]{Maria Colombo}
\address{Institute of Mathematics, EPFL, Station 8, 1015 Lausanne, Switzerland}
\email{maria.colombo@epfl.ch}

\author[Guido De Philippis]{Guido De Philippis}
\address{Dipartimento di Matematica "Tullio Levi-Civita", Via Trieste 63, Torre Archimede, Padova, Italy}
\email{guido.dephilippis@unipd.it}

\author[Carl J. P. Johansson]{Carl Johan Peter Johansson}
\address{Institute of Mathematics, EPFL, Station 8, 1015 Lausanne, Switzerland}
\email{carl.johansson@epfl.ch}

\usepackage{etoolbox}

\begin{document}

\begin{abstract}
 { 
The purpose of this note is twofold. 
First, we prove quantitative estimates for the Ambrosio-Trevisan commutator which implies propagation of logarithmic Sobolev regularity.
Second, we prove a similar estimate for the DiPerna-Lions commutator by relying on quantitative differentiation.
}
\end{abstract}

\maketitle

\tableofcontents

\section{Introduction}

In the seminal papers of DiPerna-Lions \cite{DPL89} and Ambrosio \cite{A04}, well-posedness of the transport equation

\begin{equation}\label{Transport equation}
	\partial_t u_t+b_t\cdot \nabla u_t=0\qquad\text{in}\ [0,T] \times \R^d;
	\tag{Tr}
\end{equation}
was proved for Sobolev velocity fields and for velocity fields of bounded variation.

In the paper \cite{LADT14}, Ambrosio and Trevisan have generalized the classical DiPerna-Lions theory (see \cite{DPL89, A04}) to the abstract framework of metric measure spaces. Their argument, based on a parabolic approximation procedure along with a commutator estimate, is new and interesting even in the Euclidean setting. In this note we first show how Ambrosio-Trevisan's approach can be employed to easily obtain propagation of the ``log-Sobolev regularity'' in the spirit of \cite{FL18,EBQHN21, DMCS24}. Secondly, we show how quantitative differentiation, going back to works of Dorronsoro \cite{JRD85} can be employed to propagate regularity using DiPerna-Lions's approach.

Solutions of \eqref{Transport equation} are understood in the distributional sense. To be more precise,  we assume integrability conditions of the form $b\in L^1([0,T];L^p(\setR^d;\R^d))$, $\div b\in L^1([0,T];L^p(\setR^d))$
and $u\in L^{\infty}([0,T]; L^q(\setR^d))$ with $\frac{1}{p}+\frac{1}{q}\le 1$ and we say that $(u,b)$ is a solution to \eqref{Transport equation} if the map $t\to \int u_t\phi \di x$ is absolutely continuous for any $\phi\in C^{\infty}_c(\setR^d)$ and satisfies
\begin{equation*}
\frac{\di}{\di t} \int_{\R^d} \phi u_t \di x
=\int_{\R^d} u_t\div b_t \phi\di x+\int_{\R^d} u_t b_t\cdot \nabla \phi \di x.                            
\end{equation*}
We can also assume, without loss of generality, that $t\to u_t\in L^q(\setR^d)$ is a continuous map with respect to the weak (or weak-star when $q=\infty$) topology of $L^q$.

Under the incompressibility assumption, i.e. $\div b=0$, it is easily seen that regular solutions with bounded support to \eqref{Transport equation} satisfy 
\begin{equation}\label{beta norm is preserved}
\int_{\R^d} \beta(u_t)\di x=\int_{\R^d}\beta(u_0)\di x\quad \text{$\forall t\in[0,T]$ and $\beta:\setR\to [0,\infty)$ Lipschitz with $\beta(0)=0$}.
\end{equation}
This property can be checked by observing that $\beta(u_t)$ is a solution to \eqref{Transport equation} as well
\begin{equation*}
	\partial_t\beta(u_t)+b_t\cdot \nabla \beta(u_t)=\beta'(u_t)\partial_tu_t+\beta'(u_t)b_t\cdot \nabla u_t=0,
\end{equation*}
and therefore
\begin{equation*}
	\frac{\di}{\di t} \int_{\R^d} \beta(u_t) \di x 
	=-\int_{\R^d} b_t\cdot \nabla \beta(u_t) \di x
	=-\int_{\R^d} \div (b_t\ \beta (u_t))\di x=0.	
\end{equation*}
Despite its simplicity, we can draw important consequences from \eqref{beta norm is preserved}, for instance we obtain the uniqueness property for the Cauchy problem associated to \eqref{Transport equation}.

The identity \eqref{beta norm is preserved} does not hold in general, some regularity and integrability on both $u$ and $b$ is needed. In this note we focus on the setting considered by DiPerna-Lions in their influential paper \cite{DPL89}:
\begin{equation}\label{HP}
b\in L^1([0,T];L^{\infty}(\setR^d)\cap W^{1,p}(\setR^d)),\quad
u\in L^{\infty}([0,T];L^{q}(\setR^d))\quad \text{with}\quad \frac{1}{p}+\frac{1}{q}=1.
\tag{HP}
\end{equation}
A typical approach to extend the identity \eqref{beta norm is preserved} to a weak setting, like ours, builds upon regularization arguments and commutator estimates. 
To give a more precise idea let us consider smooth mollifiers $\rho_{\eps}$ and set $u_t^{\eps}:=\rho_{\eps}\ast u_t\in L^q(\setR^d)\cap C^{\infty}(\setR^d)$ for any $t\in [0,T]$. Then
\begin{equation}\label{eq:CommutatorInIntro}
	\partial_t u_t^{\eps}+b_t\cdot \nabla u_t^{\eps}=b_t\cdot \nabla (\rho_\eps\ast u_t)-\rho_{\eps}\ast(b_t\cdot\nabla u_t) =: [b_t \cdot \nabla, \rho_{\eps} \ast] u_t,
\end{equation}
and since $u_t^{\eps}$ is smooth we have
\begin{equation*}
	\partial_t \frac{1}{2}\beta(u_t^{\eps})+b_t\cdot \nabla \beta(u_t^{\eps})=\beta'(u_t^{\eps})[b_t \cdot \nabla, \rho_{\eps} \ast] u_t,
\end{equation*}
in particular \eqref{beta norm is preserved} holds provided the commutator $[b_t \cdot \nabla, \rho_{\eps} \ast] u_t \to 0$ in $L^1$. 
This convergence has been proved in the Euclidean setting by DiPerna-Lions-Ambrosio \cite{DPL89,A04} and in the setting of metric measure spaces by Ambrosio-Trevisan \cite{LADT14}. The DiPerna-Lions commutator and Ambrosio-Trevisan commutator (both defined in Section~\ref{sec:Commutators}) are denoted by $R^{\eps}(b,u)$ and $C^{\eps}(b,u)$ respectively. In a very recent paper by Huysmans and Said \cite{LHARS24}, quantitative estimates of the DiPerna-Lions commutator were proved using tools from harmonic analysis.
This is closely connected to the earlier work of Meyer and Seis \cite{DMCS24}, which pioneered a Littlewood-Paley approach to regularity for transport equations with Sobolev velocity fields.

We reprove versions of such quantitative estimates using two new techniques: heat flow and quantitative differentiation. We begin with the results relying on heat flow.

\begin{theorem}\label{th:commutatorgeneral}
   	Let $1<p\le \infty$, $1<q<\infty$, $1<r<\infty$ satisfy
   $\frac{1}{p}+\frac{1}{q}+\frac{1}{r}=1$. 	
   For any $b\in W^{1,p}(\setR^d; \R^d)$ with $\div b=0$, $u\in L^q(\setR^d)$ and $\phi\in L^r(\setR^d)$ it holds
   \begin{equation}\label{eq:main}
   	\int_0^1 \abs{\int_{\R^d} C^{\eps}(b,u)\phi \di x}\frac{\di \eps}{\eps}
   	\leq C(d,p,q,r) \norm{\nabla_{\text{sym}}b}_{L^p}\norm{u}_{L^q}\norm{\phi}_{L^r}.
   \end{equation}
\end{theorem}

Throughout this paper, $C(\ast)$ denotes a constant depending on $\ast$, which may vary from line to line.
This quantitative estimate implies propagation of logarithmic regularity.

\begin{theorem}\label{thm:LogRegularityViaHeatFlow}
	Let $1<p\le \infty$, $1<q<\infty$, $1<r<\infty$ satisfy
	$\frac{1}{p}+\frac{1}{q}+\frac{1}{r}=1$. 
	If $u_t$ and $b_t$ solve \eqref{Transport equation} and satisfy
	\begin{equation*}
	b\in L^1([0,T];L^{\infty} \cap W^{1,p}(\setR^d; \R^d)),
	\quad
	\div b=0,
	\quad
	u\in L^{\infty}([0,T];L^{q} \cap L^r(\setR^d))
	\end{equation*}
	then
	\begin{equation*}
	[u_t]_{\log}^2-[u_0]_{\log}^2\leq C(d,p,q,r) 
	\norm{\nabla_{\text{sym}} b}_{L^1_tL^p_x}\norm{u}_{L^{\infty}_tL^q_x}\norm{u}_{L^{\infty}_tL^r_x}.
	\end{equation*}
\end{theorem}

This estimate also extends to the case $b \in L^1_t W^{1,1}_x$ with $\nabla b \in L^1_t \mathcal{H}^1_x$ where $\mathcal{H}^1$ denotes the Hardy space. We state and prove such a result in the appendix, see Theorem~\ref{thm:RegularityThmHardy}.

Finally, we show that, using quantitative differentiation, quantitative estimates for the DiPerna-Lions commutator can be obtained. Precisely, we prove the following:

\begin{theorem}\label{thm:CommutatorDorronsoro}
 Let $1 < p, q < \infty$ be such that $\frac{1}{p} + \frac{1}{q} = 1$. 
 Let $b \in W^{1,p}(\R^d; \R^d)$, $\divergence b = 0$, $u \in L^q \cap L^{\infty}(\R^d)$.
 Then, for any $\phi \in L^q \cap L^{\infty}(\R^d)$ it holds that
 \begin{equation}\label{eq:CommutatorDorronsoroConclusion}
  \left( \int_0^{\infty} \left| \int_{\R^d} R^{\eps}(b,u) \phi \, dx \right|^2 \dfrac{\di \eps}{\eps} \right)^{1/2} \leq C(d,p,q) \| b \|_{W^{1,p}} \left( \| u \|_{L^q} \| \phi \|_{L^{\infty}} + \| u \|_{L^{\infty}} \| \phi \|_{L^{q}} \right).
 \end{equation}
\end{theorem}

\begin{remark}
    Theorem~\ref{eq:CommutatorDorronsoroConclusion} also implies some propagation of regularity, 
    see Proposition~\ref{prop:PropagationOfRegularityWithDorronsoro}.
\end{remark}

While similar statements have already been obtained in \cite{LHARS24} we believe it is interesting to observe that quantitative differentiation can be applied to transport equations. We end this section by heuristic description of quantitative differentiation:  It is well-known that difference quotients of Sobolev functions converge to the gradient in $L^p$, but quantitative differentiation refines this by measuring at every point $x$ and a scale $r$, how close a function is to its best affine approximation in $B_r(x)$ and showing the sum of these deviations across dyadic scales is controlled by the $W^{1,p}$ norm. We refer the reader to Subsection~\ref{subsec:QuantDiff} for a more comprehensive overview of quantitative differentiation.

\subsection*{Acknowledgments} 
EB is supported by the ERC project MIND, Grant Agreement No.~101219635. 

GDP's research is funded by the European Research Council (ERC) through CoG   101169953 ``RISE''.\footnote{Views and opinions expressed are however those of the authors only and do not necessarily reflect those of the European Union or the European Research Council.}

MC and CJ were supported by the Swiss State Secretariat for Education, Research and Innovation (SERI) under contract number MB22.00034 through the project TENSE.  

\section{Commutators, DiPerna-Lions's and Ambrosio-Trevisan's}\label{sec:Commutators}

\subsection{DiPerna-Lions's commutator}

The strategy developed in \cite{DPL89} uses a sequence of mollifiers $\varphi_{\eps} = \eps^{-d} \varphi(\frac{x}{\eps})$ where $\varphi$ is a smooth, radial and compactly supported function. 
Throughout the manuscript, we assume $\supp \varphi \subseteq B_1$. Then the commutator takes the form
\[
 R^{\eps}(b,u)(x) = \int_{\R^d} u(y) (b(x) - b(y)) \cdot \nabla \varphi_{\eps}(x-y) \di y.
\]
After a change of variable,
\[
 R^{\eps}(b,u)(x) = \int_{\R^d} u(x + \eps h) \frac{b(x) - b(x + \eps h)}{\eps} \cdot \nabla \varphi(h) \di h.
\]
Since,
\[
 \frac{b(x + \eps h) - b(x)}{\eps} \to \nabla b (x) h \quad \text{in $L^1_{\loc}$}
\]
and $b$ is assumed to be divergence-free, the convergence of $R^{\eps}(b,u)$ to zero in $L^1$ follows.

\subsection{Ambrosio-Trevisan's commutator}\label{subsec:AT}
The strategy introduced in \cite{LADT14} follows the line of DiPerna-Lions' strategy, with a fundamental difference: Ambrosio and Trevisan use the heat semi-group to regularize solutions to \eqref{Transport equation} and exploit the PDE structure to study the commutator. 
One important outcome, is a new identity for the commutator that we are going to exploit to get log-Sobolev estimates for solutions to \eqref{Transport equation}.

Let us briefly recall very basic facts concerning the heat semi-group.
For any $u \in L^p(\setR^d)$ with $1\le p\le \infty$ we denote by
\begin{equation}\label{eq:heat convolution}
P_su(x) \coloneqq u\ast G_s(x)\quad \text{where}\quad G_s(z) \coloneqq \frac{1}{(4\pi s)^{d/2}}e^{-\frac{|z|^2}{4s}}.
\end{equation}
the heat semi-group at time $s>0$ applied to $u$. It is simple to see that
\begin{equation}\label{heat equation}
\partial_s P_su=\Delta P_su\quad \text{for any $s>0$,}\quad  \lim_{s\to 0}P_su=u
\end{equation}
where the first equation is understood in the classical sense and the limit holds for $\Leb^d$-a.e. $x\in\setR^d$ and in $L^p(\setR^d)$.
Let us also recall that $P_s$ enjoys the semi-group property
\begin{equation}\label{semigroup}
	P_{s+t}u=P_s(P_tu)
	\quad\text{for any $t,s\ge 0$},
\end{equation} 
and the regularization property
\begin{equation}\label{heat regularization}
\norm{\nabla P_sf}_{L^p} \leq \frac{C(d)}{\sqrt{s}}\norm{f}_{L^p}
\quad \text{for any}\ 1\le p\le \infty.
\end{equation}
The reader can easily check \eqref{semigroup} and \eqref{heat regularization} exploiting the identity \eqref{eq:heat convolution}.
\begin{remark}
	It is worth remarking that \eqref{heat regularization} is a very robust property that holds true in a great variety of settings, for instance Riemannian manifolds and metric measure spaces with Ricci curvature bounded below. 
\end{remark}

Similarly to \eqref{eq:CommutatorInIntro}, let $u_t$ and $b_t$ solve \eqref{Transport equation} and set $u_t^{\eps} = P_{\eps} u_t$.
Then
\begin{equation}\label{AT commutator}
\partial_t u_t^{\eps}+b_t\cdot \nabla u_t^{\eps}=b_t\cdot \nabla (P_\eps u_t)-P_{\eps}(b_t\cdot\nabla u_t) =: C^{\eps}(b_t,u_t).
\end{equation}
The new commutator $C^{\eps}(b_t,u_t)$ can be rewritten in a useful form and then easily estimated.
This is the content of the next proposition which comes from \cite{LADT14}, stated in a slightly different formalism, see \cite[Lemma 5.8]{LADT14}.

\begin{proposition}\label{prop: Ambrosio-trevisan commutator}
	Let $1\le p\le \infty$, $1\le q\le \infty$, $1\le r\le \infty$ satisfy
	$\frac{1}{p}+\frac{1}{q}+\frac{1}{r}=1$. 	
	For any $b\in W^{1,p}(\setR^d; \R^d)$ with $\div b=0$, $u\in L^q(\setR^d)$ and $\phi\in L^r(\setR^d)$ it holds
	\begin{enumerate}
		\item \label{item:RewritingOfCommutator}
		\begin{equation}\label{Commutator}
			\int_{\R^d} C^{\eps}(b,u)\phi\di x =2\int_0^{\eps}\int_{\R^d} \nabla_{\text{sym}}b \ \nabla P_s u\nabla P_{\eps-s}\phi \di x \di s;
		\end{equation}
	       \item \label{item:BoundsOnCommutator}
		\begin{equation}\label{eq:first commutator estimate}
			\abs{\int_{\R^d} C^{\eps}(b,u)\phi \di x}\leq C(d)
			\norm{\nabla_{\text{sym}}b}_{L^p}\norm{ u}_{L^q}\norm{\phi}_{L^r}.	
		\end{equation}
   \end{enumerate}
   In particular, denoting by $r'$ the conjugate of $r$ one has
   \begin{equation}\label{eq:second commutator estimate}
   	\norm{C^{\eps}(b,u)}_{L^{r'}}\leq C(d) \norm{\nabla_{\text{sym}}b}_{L^p}\norm{ u}_{L^q}.
   \end{equation}
\end{proposition}

\begin{proof}
Exploiting an interpolation argument in the spirit of Bakry-Ledoux-Gentil \cite{DBIGML14} we can write
\begin{align*}
	\int_{\R^d} C^{\eps}(b,u)\phi \di x=& \int_{\R^d} (b\cdot \nabla (P_\eps u)-P_{\eps}(b\cdot \nabla u))\phi \di x\\
	=& \int_0^\eps \int_{\R^d} \frac{\di}{\di s} P_{\eps-s}(b\cdot\nabla P_su)\phi \di x\di s\\
	=& \int_0^{\eps} \int_{\R^d} \left\lbrace P_{\eps-s}(b\cdot \nabla \Delta P_s u)-\Delta P_{\eps-s}(b\cdot \nabla P_s u)\right\rbrace \phi \di x \di s
\end{align*}
where in the last passage we used \eqref{heat equation}. Using the fact that $P_{\eps}$ and $\Delta$ are self adjoint operators in $L^2(\setR^d)$ and the integration by parts formula 
\begin{equation*}
	\int_{\R^d} g\ b\cdot \nabla f\di x=-\int_{\R^d} f\ b\cdot \nabla g
	\quad \text{for any}\ f\in \mathcal{S}(\R^d),\ g\in W^{1,\infty}(\setR^d)\cap  C^{\infty}(\setR^d)
\end{equation*}
(here we have used that $\div b=0$), we end up with
\begin{align*}
  \int_{\R^d} C^{\eps}(b,u)\phi \di x=&-\int_0^{\eps} \int_{\R^d} \left\lbrace \Delta P_s u\ b\cdot \nabla P_{\eps-s}\phi+\Delta P_{\eps-s}\phi\ b\cdot \nabla P_s u \right\rbrace \di x \di s\\
  =& 2 \int_0^{\eps} \int_{\R^d} \nabla_{\text{sym}}b \ \nabla P_s u\nabla P_{\eps-s}\phi \di x \di s,
\end{align*}
where the last identity is a consequence of the well-known equality
\begin{equation*}
  -\frac{1}{2} \int_{\R^d} \left\lbrace \Delta f\ b\cdot \nabla g+\Delta g\ b\cdot \nabla f-\div b\ \nabla f\cdot  \nabla g\right\rbrace \di x =\int_{\R^d} \nabla_{\text{sym}}b\ \nabla f\nabla g\di x
\end{equation*}
that holds true for general Riemannian manifolds.

The proof of \ref{item:BoundsOnCommutator} easily follows from \ref{item:RewritingOfCommutator}	 by applying the H\"older inequality and using \eqref{heat regularization}.
We have indeed
\begin{align*}
	\abs{\int_0^{\eps} \int_{\R^d} \nabla_{\text{sym}}b \ \nabla P_s u\nabla P_{\eps-s}\phi \di x \di s }
	\le & \int_0^{\eps}\norm{\nabla_{\text{sym}}b}_{L^p}\norm{\nabla P_s u}_{L^q}\norm{\nabla P_{\eps-s}\phi}_{L^r}\di s\\
	\leq C(d) & \norm{\nabla_{\text{sym}}b}_{L^p}\norm{ u}_{L^q}\norm{\phi}_{L^r} \int_0^\eps\frac{1}{\sqrt{s}}\ \frac{1}{\sqrt{\eps-s}}\di s\\
	\leq C(d) & \norm{\nabla_{\text{sym}}b}_{L^p}\norm{ u}_{L^q}\norm{\phi}_{L^r}.
\end{align*}
Then, \eqref{eq:second commutator estimate} follows from \ref{item:BoundsOnCommutator} by duality.
\end{proof}

\begin{remark}
	One can relax the assumption $\div b=0$ to $\div b\in L^p(\setR^d)$ and obtain similar results. For the sake of simplicity we prefer to consider just the incompressible case.
\end{remark}

An approximation argument, along with Proposition~\ref{prop: Ambrosio-trevisan commutator}, allows us to prove 
\begin{equation}\label{z1}
   \lim_{\eps\to 0}\int_0^T\norm{C^{\eps}(b_t,u_t)}_{L^1}\di t=0
\end{equation}
when $u_t$ and $b_t$ satisfies \eqref{HP} and $\divergence b_t = 0$. 

Observe that, the same argument in Proposition~\ref{prop: Ambrosio-trevisan commutator} with $r=\infty$ applied to $C^{\eps}(b_t, P_{\delta}u_t)$ gives
\begin{align*}
\abs{\int_0^{\eps} \int_{\R^d} \nabla_{\text{sym}}b_t \ \nabla P_{s+\delta} u_t\nabla P_{\eps-s}\phi \di x \di s }
\le & \int_0^{\eps}\norm{\nabla_{\text{sym}}b_t}_{L^p}\norm{\nabla P_{s+\delta} u_t}_{L^q}\norm{\nabla P_{\eps-s}\phi}_{L^\infty}\di s\\
\leq C(d) & \norm{\nabla_{\text{sym}}b_t}_{L^p}\norm{ u_t}_{L^q}\norm{\phi}_{L^\infty} \int_0^\eps\frac{1}{\sqrt{s+\delta}}\ \frac{1}{\sqrt{\eps-s}}\di s\\
\leq C(d) & \norm{\nabla_{\text{sym}}b_t}_{L^p}\norm{ u_t}_{L^q}\norm{\phi}_{L^\infty} \frac{\sqrt{\eps}}{\sqrt{\delta}}
\end{align*}
and thus
\begin{equation*}
	\norm{C^{\eps}(b_t,P_{\delta}u_t)}_{L^1}\leq C(d) \frac{\sqrt{\eps}}{\sqrt{\delta}} \norm{\nabla_{\text{sym}}b_t}_{L^p}\norm{ u_t}_{L^q}.
\end{equation*}
Using \eqref{eq:second commutator estimate} with $r= \infty$ we obtain
\begin{align*}
	\int_0^T\norm{C^{\eps}(b_t,u_t)}_{L^1}\di t
	\le & \int_0^T\norm{C^{\eps}(b_t,P_{\delta}{u_t)}}_{L^1}\di t
    + \int_0^T\norm{C^{\eps}(b_t,u_t-P_{\delta}u_t)}_{L^1}\di t\\
    \leq C(d) & \frac{\sqrt{\eps}}{\sqrt{\delta}} \norm{\nabla_{\text{sym}} b}_{L^1_tL^p_x}\norm{u}_{L^{\infty}_tL^q_x}
    +\norm{\nabla_{\text{sym}} b}_{L^1_tL^p_x}\norm{u-P_{\delta } u}_{L^{\infty}_tL^q_x}.
\end{align*}
Choosing $\delta=\sqrt{\eps}$ and letting $\eps\to 0$ we get \eqref{z1}.

\section{Improved estimates and regularity via Ambrosio-Trevisan's commutator}

\subsection{Proof of Theorem~\ref{th:commutatorgeneral}}

We recall the definition and the main properties of the square function associated to the heat semi-group.

\begin{lemma}
	[Integrability of the Square function]\label{lemma:square function integrability}
	Let us set
	\begin{equation}\label{square function}
	S f \coloneqq \left( \int_0^{\infty}|\nabla P_s f|^2\di s   \right)^{1/2}\quad \text{for $f\in \mathcal{S}(\R^d)$}.
	\end{equation}
	Then we have 
	\begin{equation*}
	\norm{Sf}_{L^p}\leq C(d,p) \norm{f}_{L^p}
	\quad \text{for any }1<p<\infty.
	\end{equation*}	
\end{lemma}
\begin{proof}
	This result is standard in harmonic analysis and builds upon the Calder\'on-Zygmund theory, we refer to \cite[Chapter 2]{EMS70} or \cite[Sections 6.3 \& 6.4]{EMS93} for a proof.
	For sake of completeness, here we just present the estimate (actually the identity) in the particular case $p=2$ that can be performed exploiting the PDE structure.
	
	Integrating by parts and using \eqref{heat equation} we get
	\begin{align*}
	\norm{Sf}_{L^2}^2 &= \int_0^{\infty}\int_{\R^d} |\nabla P_s f|^2\di x \di s
	=-\int_0^{\infty}\int_{\R^d} P_s f\ \Delta P_s f\di x \di s\\
	&=-\int_0^{\infty}\int_{\R^d} P_s f \partial_s P_s f \di x\di s
	=-\int_0^{\infty}\frac{1}{2}\frac{\di}{\di s}\int_{\R^d} (P_s f)^2 \di x \di s\\
	&=\frac{1}{2}\norm{f}_{L^2}^2,
	\end{align*}
	where in the last step we used the following property of the heat semi-group
	\begin{equation*}
	\lim_{s\to \infty}\norm{P_s f}_{L^2}=0
	\quad \text{for any }f\in L^2(\setR^d).
	\end{equation*}
\end{proof}

\begin{proof}[Proof of Theorem~\ref{th:commutatorgeneral}]
	The starting point is the commutator identity \eqref{Commutator}:
    \begin{equation*}
    \frac{1}{2} \abs{\int_{\R^d} C^{\eps}(b,u)\phi\di x}=\abs{\int_0^{\eps} \int_{\R^d} \nabla_{\text{sym}}b \ \nabla P_s u\nabla P_{\eps-s}\phi \di x \di s}
    \le \int_0^{\eps}\int_{\R^d} |\nabla_{\text{sym}}b| |\nabla P_s u||\nabla P_{\eps-s}\phi| \di x \di s.
    \end{equation*}
    Our conclusion follows from the claim below applied to $h=|\nabla_{\text{sym}} b|$, $f = u$, $g = \phi$.
    
    \medskip
    \textbf{Claim:}
    Let $f,g,h\in \mathcal{S}(\R^d)$ then
    \begin{equation*}
    C(h,f,g) \coloneqq \int_0^{\infty} \int_0^{\eps} \int_{\R^d} |h| |\nabla P_s f||\nabla P_{\eps-s}g|\di x\di s \frac{\di \eps}{\eps}
   \leq C(d,p,q,r) \norm{h}_{L^p}\norm{f}_{L^q}\norm{g}_{L^r}
   \end{equation*}
   for any $1< p\le \infty$, $1<q<\infty$, $1<r<\infty$ satisfying $\frac{1}{p}+\frac{1}{q}+\frac{1}{r}=1$.
   \medskip
    By a change of variables, we get
	\begin{equation}\label{z3}
	C(h,f,g) \leq \int_{\R^d} |h| \int_0^{\infty}|\nabla P_sf|\int_0^{\infty}|\nabla P_{\eps} g|\frac{\di\eps}{\eps+s}\di s\di x.
	\end{equation}
	Let us now consider the linear operator
	\begin{equation*}
		H \phi (s):=\int_0^{\infty} \phi(\eps)\frac{\di \eps}{\eps+s}
		\quad \text{for any measurable function } \phi:[0,\infty)\to [0,\infty), 
	\end{equation*}
	and observe that 
	\begin{equation}\label{z2}
		\norm{H\phi}_{L^p((0,\infty))}\leq C(p) \norm{\phi}_{L^p((0,\infty))}
		\quad \text{for }1<p<\infty.
	\end{equation}
	The inequality \eqref{z2} can be checked noticing that, for $1\le p<\infty$, the H\"older inequality gives
	\begin{equation*}
		H\phi(s)\leq C(p) \norm{\phi}_{L^p} s^{-1/p}
		\implies \norm{H\phi}_{L^{p,\infty}((0,\infty))} \leq C(p) \norm{\phi}_{L^p((0,\infty))}
	\end{equation*}
	and therefore \eqref{z2} follows from the Marcinkiewicz interpolation theorem.
	Applying \eqref{z2} with $\phi(\eps)=|\nabla P_{\eps} g|$ and $p=2$ we get
	\begin{equation}\label{z4}
		\int_0^{\infty} |\nabla P_s f|\int_0^{\infty} |\nabla P_{\eps}g| \frac{\di\eps}{\eps+s}\di s
		\leq C \left( \int_0^{\infty}|\nabla P_s f|^2\di s   \right)^{1/2}\left( \int_0^{\infty}|\nabla P_s g|^2\di s   \right)^{1/2} = C Sf Sg,
	\end{equation}
   where $Sf$ denotes the square functions already introduced in Lemma~\ref{lemma:square function integrability}.
   We can now easily conclude using the H\"older inequality and Lemma~\ref{lemma:square function integrability}
   \begin{equation*}
		C(f,g,h) \leq C \int_{\R^d} |h| Sf Sg \di x
		\leq C \norm{h}_{L^p}\norm{Sf}_{L^q}\norm{Sg}_{L^r}
		\leq C(d,p,q,r) \norm{h}_{L^p}\norm{f}_{L^q}\norm{g}_{L^r}.
	\end{equation*}
\end{proof}

\begin{remark}\label{remark:variation commutator estimate}
	Following the proof of Theorem~\ref{th:commutatorgeneral} one can easily show 
	\begin{equation*}
	\int_0^1 \abs{\int_{\R^d} P_{\eps} C^{\eps}(b,u)\phi \di x}\frac{\di \eps}{\eps}
	\leq C(d,p,q,r) \norm{\nabla_{\text{sym}}b}_{L^p}\norm{u}_{L^q}\norm{\phi}_{L^r}.
	\end{equation*}
	This variant of \eqref{eq:main} will play a role in the sequel.	
\end{remark}

\subsection{Proof of Theorem~\ref{thm:LogRegularityViaHeatFlow}}

As we have anticipated, the stronger commutator estimate in Theorem~\ref{th:commutatorgeneral} gives a new proof of the propagation of ``log-Sobolev'' regularity for solutions of \eqref{Transport equation}. 

Let us start by introducing a semi-norm aiming to measure the ``derivative of logarithmic order'' of a given function
\begin{equation}\label{eq:log-seminorm}
[u]_{\log}^2 \coloneqq \int_0^1 \norm{P_{\eps}u-u}_{L^2}^2\frac{\di \eps}{\eps}.
\end{equation}
In the next lemma we will see that \eqref{eq:log-seminorm} is actually equivalent to the semi-norm proposed in \cite{EBQHN21} and \cite{FL18}.

\begin{lemma}
	Let $u\in \mathcal{S}(\R^d)$. We have
	\begin{equation}\label{eq:equivalence norms 1}
	\frac{1}{C} \int_{|\xi|\ge 1} \log |\xi| |\hat u(\xi)|^2\di \xi
	\leq [u]_{\log}^2
	\leq C \left( \int_{|\xi|\ge 1} \log |\xi| |\hat u(\xi)|^2\di \xi + \int_{|\xi|\le 1}|\hat u(\xi)|^2\di \xi \right)
	\end{equation}
	In particular, up to a dimensional constant $C(d)$, we have the equivalence between
	\begin{equation}\label{eq:equivalence norms 2}
	\norm{u}_{L^2}^2+[u]_{\log}^2 \quad \text{and} \quad
	 \norm{u}_{L^2}^2+\int_{B_{1/3}}\int_{\setR^d}\frac{|u(x+h)-u(x)|^2}{|h|^d}\di x\di h.
	\end{equation}	
\end{lemma}

\begin{proof}
	Using the Plancherel formula and the identity $\widehat{ P_tu}(\xi)=e^{-4 \pi^2 |\xi|^2 t}\hat u(\xi)$ we obtain
	\begin{align*}
	[u]^2_{\log}=& \int_0^1\int_{\R^d} |\hat u(\xi)|^2(1-e^{-4 \pi^2\eps|\xi|^2})^2\di \xi \frac{\di \eps}{\eps}\\
	=&\int_{\R^d} |\hat u(\xi)|^2\int_0^1 (1-e^{-4 \pi^2\eps|\xi|^2})^2\frac{\di \eps}{\eps} \di \xi\\
	=&\int_{\R^d} |u \hat(\xi)|^2\int_0^{4 \pi^2 |\xi|^2}(1-e^{-s})^2 \frac{\di s}{s}\di \xi
	\end{align*}
	where in the last passage we changed variables according to $4 \pi^2 |\xi|^2 \eps = s$. In order to conclude one has to check the simple inequalities
	\begin{equation*}
	\frac{1}{C} \chi_{\{|\xi|\ge 1\}} \log |\xi| \leq	\int_0^{|\xi|^2}(1-e^{-s})^2 \frac{\di s}{s}
	 \leq C ( \chi_{\{|\xi|\ge 1\}}\log |\xi|+1 ).
	\end{equation*}
	
	The proof of \eqref{eq:equivalence norms 2} follows from \eqref{eq:equivalence norms 1} and the equivalence, up to a dimensional constant $C(d)$ of
	\begin{equation*}
	\norm{u}_{L^2}^2+\int_{|\xi|\ge 1}\log |\xi||\hat u(\xi)|^2\di \xi \text{ and } \norm{u}_{L^2}^2+\int_{B_{1/3}}\int_{\setR^d}\frac{|u(x+h)-u(x)|^2}{|h|^d}\di x\di h
	\end{equation*}
	that has been proved in \cite[Lemma 3.1]{FL18}.
\end{proof}

\begin{proof}[Proof of Theorem~\ref{thm:LogRegularityViaHeatFlow}]
	Let us assume, without loss of generality, that $u_0\in \mathcal{S}(\R^d)$ and $b_t\in C^{\infty}_c(\setR^d;\setR^d)$ uniformly in time. Applying classical results we have that $u_t \in \mathcal{S}(\R^d)$ for any $t\in [0,T]$. Recalling \eqref{AT commutator} we get
	\begin{equation}\label{z6}
	\partial_t \frac{1}{2}|P_\eps u_t-u_t|^2+ b_t\cdot\nabla  \frac{1}{2}|P_{\eps}u_t-u_t|^2=C^{\eps}(b_t,u_t)\cdot (P_{\eps}u_t-u_t).
	\end{equation}
	Integrating \eqref{z6} in space and time we obtain
	\begin{align}
    \begin{split}\label{eq:EnergyBalancePlusCommutator}
	\norm{P_\eps u_t-u_t}_{L^2}^2&-\norm{P_\eps u_0-u_0}_{L^2}^2\\
	=& 2 \int_0^t \int_{\R^d} C^{\eps}(b_s,u_s)\cdot (P_{\eps}u_s-u_s)\di x \di s\\
	\leq & 2 \int_0^T\abs{\int_{\R^d}  P_\eps C^{\eps}(b_s,u_s) u_s\di x}\di s+ 2 \int_0^T \abs{\int_{\R^d} C^{\eps}(b_s,u_s) u_s\di x}\di s
    \end{split}
	\end{align}
	that leads to
	\begin{equation}
	[u_t]_{\log}^2\le [u_t]_{\log}^2
	+2 \int_0^T \int_0^1 \abs{\int_{\R^d}  P_\eps C^{\eps}(b_s,u_s)\ u_s\di x}\frac{\di\eps}{\eps}\di s
	+ 2 \int_0^T \int_0^1 \abs{\int_{\R^d} C^{\eps}(b_s,u_s)\ u_s\di x}\frac{\di\eps}{\eps}\di s.
	\end{equation}
	We eventually apply Theorem~\ref{th:commutatorgeneral} and Remark~\ref{remark:variation commutator estimate} to conclude the proof.
\end{proof}

\section{Improved estimates via quantitative differentiation}

\subsection{Quantitative differentiation}\label{subsec:QuantDiff}

In 1985, Dorronsoro \cite{JRD85} introduced quantitative differentiation which roughly states that asymptotically and quantitatively the best affine approximation of a Sobolev function over a ball converges to the gradient as the radius of the ball vanishes. In practice, the implementation consists of a multiscale $\beta$-number approach reminiscent of Jones' $\beta$-numbers \cite{PJ90}.
Since the work of Dorronsoro, quantitative differentiation has been further developed in multiple directions, see for instance \cite{GDSS91, JAXT15}.

\begin{definition}[{\cite[p. 22]{JRD85}}]
 Let $f \in L^1_{\loc}(\R^d)$ and $B \subseteq \R^d$ a ball.
 We call $P_B^k f \colon \R^d \to \R$ the unique polynomial of degree $k$ such that
 \[
  \int_B \left( f(y) - P^k_B f(y) \right) y^{\gamma} \di y = 0
 \]
 for all $n$-tuples $\gamma = (\gamma_1, \ldots, \gamma_d)$ with $|\gamma| = \gamma_1 + \ldots + \gamma_d \leq k$.
 Additionally, we define $\Omega_f^k \colon \R^d \times \R \to \R$ as
 \[
  \Omega_f^k(x,t) = \sup \left\{ \fint_B |f - P_B^k f| \di z : x \in B, |B| = t^d \right\}.
 \]
\end{definition}

\begin{theorem}[{\cite[Theorem 2]{JRD85}}]\label{thm:QuantDiffDorronsoro}
 Let $f \in L^p(\R^d)$ for some $1 < p < \infty$. Then $f \in W^{k, p}(\R^d)$ with $k \in \N_{\geq 1}$ if and only if the function $G_k f \colon \R^d \to \R$ defined by
 \[
  G_k f (x) \coloneqq \left( \int_0^{\infty} (t^{-k} \Omega_f^k(x,t))^2 t^{-1} \di t \right)^{1/2}
 \]
 belongs to $L^p(\R^d)$. Furthermore, it then holds that
 \[
  C(d,k,p)^{-1} \| f \|_{W^{k, p}} \leq \| f \|_{L^p} + \| G_k f \|_{L^p} \leq C(d,k,p) \| f \|_{W^{k, p}}.
 \]
\end{theorem}

\begin{remark}\label{rmk:ExplicitExpressionDorronsoro}
 The polynomial $P^1_B f$ can be determined explicitly and is given by
 \begin{equation}
  P^1_B f (y) = \fint_B f \, dx + \nabla^{\app}_B f \cdot (y - x_B)
 \end{equation}
 where $x_B$ is the center of $B$ and
 \begin{equation}\label{eq:ApproximateGradientScalarMap}
  \nabla^{\app}_B f \coloneqq \dfrac{C(d)}{\diam(B)^2} \fint_B f(x) (x- x_B) \di x.
 \end{equation}
\end{remark}

\begin{lemma}\label{lemma:EstimateOnApproxGradient}
 For any $f \in W^{1,1}_{\loc}(\R^d)$ and any ball $B$ 
 \begin{equation}\label{eq:EstimateOnApproxGradient}
     |\nabla^{\app}_B f| \leq C(d) \MF | \MF |\nabla f| |(x_B). 
 \end{equation}
\end{lemma}
\begin{proof}
 By Equation~\eqref{eq:ApproximateGradientScalarMap}, for a.e. $y \in B$
  \begin{equation}
  \nabla^{\app}_B f = \dfrac{C(d)}{\diam(B)^{2+d}} \int_B ( f(x) - f(y) ) (x- x_B) \di x.
 \end{equation}
 Therefore,
 \begin{equation}
  | \nabla^{\app}_B f | \leq \dfrac{C(d)}{\diam(B)^{1+d}} \int_B | f(x) - f(y) | \di x.
 \end{equation}
 Using $|f(x) - f(y)| \leq C(d) (\MF |\nabla f|(x) + \MF |\nabla f|(y)) |x-y|$ for a.e. $x,y \in \R^d$, we deduce
 \begin{equation}
  | \nabla^{\app}_B f | \leq C(d) \fint_B (\MF |\nabla f|(x) + \MF |\nabla f|(y)) \di x \leq C(d) \Big( \MF | \MF |\nabla f| | (x_B) + \MF |\nabla f|(y) \Big).
 \end{equation}
 Integrating with respect to $y$ gives Equation~\eqref{eq:EstimateOnApproxGradient}.
\end{proof}

\begin{remark}
 From Remark~\ref{rmk:ExplicitExpressionDorronsoro}, it follows that for vector-valued maps $b \colon \R^d \to \R^d$,
 \begin{equation}\label{eq:ApproximateGradientVectorMap}
  \nabla_B^{\app} b = \frac{C(d)}{\diam(B)^2} \fint_B b(x) \otimes (x - x_B) \di x.
 \end{equation}
\end{remark}

\begin{lemma}
 If $b \in W^{1,1}_{\loc}(\R^d; \R^d)$ is divergence-free, then $\Tr(\nabla_B^{\app} b) = 0$ for all balls $B$.
\end{lemma}

\begin{proof}
 By Equation~\eqref{eq:ApproximateGradientVectorMap}, 
  \begin{equation}
  \Tr(\nabla_B^{\app} b) = \frac{C(d)}{\diam(B)^2} \fint_B b(x) \cdot \nabla \left( \frac{|x - x_B|^2}{2} \right) \di x.
 \end{equation}
 Integrating by parts, using $\divergence b \equiv 0$ and $\frac{|x-x_B|^2}{2} = \frac{\diam(B)^2}{8}$ on $\partial B$, we deduce $\Tr(\nabla_B^{\app} b) = 0$.
\end{proof}

\subsection{Proof of Theorem~\ref{thm:CommutatorDorronsoro}}
We start with an attempt to prove Theorem~\ref{thm:CommutatorDorronsoro} which shows the main ideas and the required refinements to obtain a proof. The quantity of interest is rewritten as
\begin{equation}
 \int_{\R^d} R^{\eps}(b, u) \phi \di x = \int_{\R^d} \int_{\R^d} \phi(x) u(y) (b(x) - b(y)) \cdot \nabla \varphi_{\eps}(x-y) \di y \di x.
\end{equation}
By adding and subtracting $P_{B_{\eps}(y)}b (x)$ from $b(x) - b(y)$, quantitative differentiation can be used to estimate part of the expression. The remaining expression is
\begin{equation}
  \int_{\R^d} \int_{\R^d} \phi(x) u(y) \left(P_{B_{\eps}(y)}b(x) - b(y) \right) \cdot \nabla \varphi_{\eps}(x-y) \di y \di x.
\end{equation}
Clearly, it is not straightforward to estimate this term using quantitative differentiation.
However, one could add and subtract $P_{B_{\eps}(x)}b (y)$ and use quantitative differentiation again. This leads to an additional term containing $P_{B_{\eps}(y)}b(x) - P_{B_{\eps}(x)}b(y)$. This is one of the ideas of the proof of Theorem~\ref{thm:CommutatorDorronsoro}. To obtain a simple expression for the last additional term, we introduce a new variable $\bar{x}$ which we use as center for the affine approximations when using quantitative differentiation, since then $P_{B_{\eps}(\bar{x})}b(x) - P_{B_{\eps}(\bar{x})}b(y) = (\nabla^{\app}_{B_{2\eps}(\bar{x})} b)(x-y)$. To estimate this additional term, we use Lemma~\ref{lemma:CZLemma} below and the fact that $\nabla^{\app}_{B_{2\eps}(\bar{x})} b$ is trace-free. Indeed, since $\nabla^{\app}_{B_{2\eps}(\bar{x})} b$ is trace-free, $u(y)$ can be replaced with $u(y) - f(x, \eps)$ and by choosing $f$ appropriately, we can apply Lemma~\ref{lemma:CZLemma} to $u(y) - f(x, \eps)$.

\begin{lemma}[Integrability of the Square function II]\label{lemma:CZLemma}
 Let $\psi \in C^{\infty}_c(\R^d)$ be a function such that 
 \begin{equation}\label{eq:ConditionForCZLemma}
  \int_{B_r(0) \setminus B_s(0)} \psi \di x = 0 \quad \forall \, 0 < s < r < \infty.
 \end{equation}
 Then the operator $\Psi$ defined by 
 \begin{equation}
  \Psi f (x) \coloneqq \left( \int_0^{\infty} |f \ast \psi_{\eps} |^2(x) \frac{\di \eps}{\eps} \right)^{1/2}
 \end{equation}
 satisfies $\| \Psi f \|_{L^p} \leq C(d,p, \psi) \| f \|_{L^p}$ for any $1 < p < \infty$.
\end{lemma}

\begin{proof}
 We carry out the proof only in the case $p = 2$. 
 The remaining cases follow from \cite[Theorem 5.6.1]{LG14-Classical} using that $(x, \eps) \mapsto \psi_{\eps}(x)$ defines an $L^2(\frac{d \eps}{\eps})$-valued Calder\'{o}n-Zygmund kernel. By Plancherel's theorem
 \begin{equation}
  \| \Psi f \|_{L^2}^2 = \int_{\R^d} \int_0^{\infty} |f \ast \psi_{\eps}|^2(x) \dfrac{\di \eps}{\eps} \di x = \int_0^{\infty} \int_{\R^d} |\widehat{f \ast \psi_{\eps}}|^2(\xi) \di \xi \dfrac{\di \eps}{\eps}.
 \end{equation}
 Then $\widehat{f \ast \psi_{\eps}} (\xi) = \hat{f}(\xi) \hat{\psi}(\eps \xi)$ and since $\psi \in C^{\infty}_c(\R^d)$ with mean zero, we get $|\hat{\psi}(\xi)| \leq C(d, \psi) \min (\sqrt{|\xi|}, |\xi|^{-1})$. Thus,
 \begin{equation}
  \int_0^{\infty} |\hat{\psi}(\eps \xi)|^2 \dfrac{d \eps}{\eps} \leq C(d, \psi)
 \end{equation}
 from which we conclude $\| \Psi f \|_{L^2} \leq C(d, \psi) \| f \|_{L^2}$.
\end{proof}

\begin{proof}[Proof of Theorem~\ref{thm:CommutatorDorronsoro}]
Observe
\begin{equation}\label{eq:DiPernaLionsCommutatorIntegratedAgainstTest}
\begin{split}
 \int_{\R^d} R^{\eps}(b, u)(x) \phi(x) \di x &= \int_{\R^d} \int_{\R^d} \varphi_{\eps}(x - \bar{x}) R^{\eps}(b, u)(x) \phi(x) \di x \di \bar{x} \\
 &= \int_{\R^d} \int_{\R^d} \int_{\R^d} \varphi_{\eps}(x - \bar{x}) \phi(x) u(y) (b(x) - b(y)) \cdot \nabla \varphi_{\eps}(x-y) \di y \di x \di \bar{x}.
\end{split}
\end{equation}
Adding and subtracting,
\[
 b(x) - b(y) =  b(x) - P_{B_{2\eps}(\bar{x})} b(x) + P_{B_{2\eps}(\bar{x})} b(x) - P_{B_{2\eps}(\bar{x})} b(y) + P_{B_{2\eps}(\bar{x})} b(y) - b(y).
\]
Since $P_{B_{2\eps}(\bar{x})} b(x) - P_{B_{2\eps}(\bar{x})} b(y) = (\nabla^{\app}_{B_{2\eps}(\bar{x})} b)(x-y)$, Equation~\eqref{eq:DiPernaLionsCommutatorIntegratedAgainstTest} is decomposed into three terms
\begin{align*}
 \int_{\R^d} R^{\eps}(b, u)(x) \phi(x) \di x &= \underbrace{\int_{\R^d} \int_{\R^d} \int_{\R^d} \varphi_{\eps}(x - \bar{x}) \phi(x) u(y) (b(x) - P_{B_{2\eps}(\bar{x})} b(x)) \cdot \nabla \varphi_{\eps}(x-y) \di y \di x \di \bar{x}}_{= I(\eps)} \\
 &\qquad + \underbrace{\int_{\R^d} \int_{\R^d} \int_{\R^d} \varphi_{\eps}(x - \bar{x}) \phi(x) u(y) (\nabla^{\app}_{B_{2\eps}(\bar{x})} b)(x-y) \cdot \nabla \varphi_{\eps}(x-y) \di y \di x \di  \bar{x}}_{= II(\eps)} \\
 &\qquad + \underbrace{\int_{\R^d} \int_{\R^d} \int_{\R^d} \varphi_{\eps}(x - \bar{x}) \phi(x) u(y) (P_{B_{2\eps}(\bar{x})} b(y) - b(y)) \cdot \nabla \varphi_{\eps}(x-y) \di y \di x \di \bar{x}.}_{= III(\eps)}
\end{align*}
The first and third term are treated similarly using quantitative differentiation (see Subsection~\ref{subsec:QuantDiff}):
\begin{align*}
 |I(\eps)| &\leq C(d) \| \phi \|_{L^{\infty}} \int_{\R^d} \fint_{B_{2\eps}(\bar{x})} \dfrac{|b(x) - P_{B_{2\eps}(\bar{x})} b(x)|}{2 \eps} \di x \MF(u)(\bar{x}) \di \bar{x} \\
 &\leq C(d) \| \phi \|_{L^{\infty}} \int_{\R^d} \eps^{-1} \Omega^1_b(\bar{x}, c(d) \eps) \MF(u)(\bar{x}) \di \bar{x}; \\
 |III(\eps)| &\leq C(d) \| u \|_{L^{\infty}} \int_{\R^d} \eps^{-1} \Omega^1_b(\bar{x}, c(d) \eps) \MF(\phi)(\bar{x}) \di \bar{x}.
\end{align*}
Thus, using Minkowski's integral inequality and Theorem~\ref{thm:QuantDiffDorronsoro},
\begin{equation}\label{eq:FullFirstTerm}
\begin{split}
 \left( \int_0^{\infty} |I(\eps)|^2 \,\dfrac{\di \eps}{\eps} \right)^{1/2} &\leq C(d) \| \phi \|_{L^{\infty}} \int_{\R^d} \left( \int_0^{\infty} \big( \eps^{-1} \Omega^1_b(\bar{x}, c(d) \eps) \big)^2 \dfrac{d \eps}{\eps} \right)^{1/2} \MF(u)(\bar{x}) \di \bar{x} \\
 &\leq C(d) \| \phi \|_{L^{\infty}} \int_{\R^d} G_1 b(\bar{x}) \MF(u)(\bar{x}) \di \bar{x} \\
 &\leq C(d,p,q) \| \phi \|_{L^{\infty}} \| b \|_{W^{1,p}} \| u \|_{L^{q}};
 \end{split}
\end{equation}
\begin{equation}\label{eq:FullThirdTerm}
 \left( \int_0^{\infty} |III(\eps)|^2 \,\dfrac{d \eps}{\eps} \right)^{1/2} \leq C(d,p,q) \| u \|_{L^{\infty}} \| b \|_{W^{1,p}} \| \phi \|_{L^{q}}.
\end{equation}
For the second term, let $\psi \colon \R^d \to \R$ be a function to be determined. Then, since $\Tr(\nabla^{\app}_{B_{\eps}(\bar{x})} b) = 0$,

\begin{equation}\label{eq:DecompositionOfTheSecondTerm}
\begin{split}
 II(\eps) &= \int_{\R^d} \int_{\R^d} \int_{\R^d} \varphi_{\eps}(x - \bar{x}) \phi(x) \left( u(y) - (u \ast \psi_{\eps})(x) \right) (\nabla^{\app}_{B_{2 \eps}(\bar{x})} b)(x-y) \cdot \nabla \varphi_{\eps}(x-y) \di y \di x \di \bar{x} \\
 &= \sum_{1 \leq i , j \leq d} \underbrace{\int_{\R^d} \int_{\R^d} \int_{\R^d} \varphi_{\eps}(x - \bar{x}) \phi(x) \left( u(y) - (u \ast \psi_{\eps})(x) \right) (\nabla^{\app}_{B_{2 \eps}(\bar{x})} b)_{ji} (x_i-y_i) \partial_j \varphi_{\eps}(x-y) \di y \di x \di \bar{x}.}_{=: II^{ij}(\eps)}
\end{split}
\end{equation}

For any $\eps > 0$, we define $\Psi_{\eps}^{ij} \colon \R^d \to \R$ by $\Psi_{\eps}^{ij}(z) = z_i \partial_j \varphi_{\eps}(z)$ and note that $\Psi_{\eps}^{ij}(z) = \eps^{-d} \Psi_{1}^{ij}(\frac{z}{\eps})$. 
Moreover, note that
\[
 \int_{B_r(0) \setminus B_s(0)} \Psi_{1}^{ij}(z) \di z = \delta_{ij} \int_{B_r(0) \setminus B_s(0)} \Psi_{1}^{11}(z) \di z \quad \forall \, 0 < s < r < \infty.
\]
Call $m = \int_{\R^d} \Psi_{1}^{11}(z) \, dz$ and observe that each one of the terms in Equation~\eqref{eq:DecompositionOfTheSecondTerm} equals
\begin{equation}\label{eq:EachTermInTheSecondTerm}
\begin{split}
 II^{ij}(\eps)&=\int_{\R^d} \int_{\R^d} \int_{\R^d} (\nabla^{\app}_{B_{2\eps}(\bar{x})} b)_{ji} \varphi_{\eps}(x - \bar{x}) \phi(x) \left( u(y) - (u \ast \psi_{\eps})(x) \right) \Psi_{\eps}^{ij} (x-y) \di y \di x \di \bar{x} \\
 &= \int_{\R^d} \int_{\R^d} (\nabla^{\app}_{B_{2\eps}(\bar{x})} b)_{ji} \varphi_{\eps}(x - \bar{x}) \phi(x) \left( (u \ast \Psi_{\eps}^{ij})(x) - m \delta_{ij} (u \ast \psi_{\eps})(x) \right) \di x \di \bar{x} \\
 &= \int_{\R^d} \int_{\R^d} (\nabla^{\app}_{B_{2\eps}(\bar{x})} b)_{ji} \varphi_{\eps}(x - \bar{x}) \phi(x) \left( (u \ast (\Psi_{1}^{ij} - m \delta_{ij} \psi)_{\eps})(x) \right) \di x \di \bar{x}.
\end{split}
\end{equation}
Since by Lemma~\ref{lemma:EstimateOnApproxGradient} $|(\nabla^{\app}_{B_{2\eps}(\bar{x})} b)_{ji}| \leq C(d) \MF|\MF|\nabla b||(\bar{x})$,
\[
 \int_{\R^d} (\nabla^{\app}_{B_{2 \eps}(\bar{x})} b)_{ij} \varphi_{\eps}(x - \bar{x}) \di \bar{x} \leq C(d) \int_{\R^d} \MF|\MF|\nabla b||(\bar{x}) \varphi_{\eps}(x - \bar{x}) \di \bar{x} \leq C(d) \MF|\MF|\MF|\nabla b|||({x}).
\]
By plugging this into Equation~\eqref{eq:EachTermInTheSecondTerm}, each term in Equation~\eqref{eq:EachTermInTheSecondTerm} is bounded by
\begin{align*}
 |II^{ij}(\eps)| &\leq C(d) \| \phi \|_{L^{\infty}} \int_{\R^d} g(x) \left| \left( (u \ast (\Psi_{1}^{ij} - m \delta_{ij} \psi)_{\eps})(x) \right) \right| \di x.
\end{align*}
with $g(x) = \MF|\MF|\MF|\nabla b||| (x)$.
Thus, using Minkowski's integral inequality,
\begin{align*}
 \left( \int_0^{\infty} |II^{ij}(\eps)|^2 \dfrac{\di \eps}{\eps} \right)^{1/2} &\leq C(d) \| \phi \|_{L^{\infty}} \int_{\R^d} g(x) \left( \int_0^{\infty} \left| \left( (u \ast (\Psi_{1}^{ij} - m \delta_{ij} \psi)_{\eps})(x) \right) \right|^2 \dfrac{\di \eps}{\eps} \right)^{1/2} \di x.
\end{align*}
Select the function $\psi$ such that
\begin{equation}
 \int_{B_r(0) \setminus B_s(0)} \psi(z) \di z = \dfrac{1}{m} \int_{B_r(0) \setminus B_s(0)} \Psi_1^{11}(z) \di z \quad \forall \, 0 < s < r < \infty
\end{equation}
 so that each function $\Psi_{1}^{ij} - m \delta_{ij} \psi$ satisfies Equation~\eqref{eq:ConditionForCZLemma}.
 Therefore, by Lemma~\ref{lemma:CZLemma}
 \begin{equation}\label{eq:FullSecondTerm}
 \begin{split}
  \left( \int_0^{\infty} |II(\eps)|^2 \dfrac{\di \eps}{\eps} \right)^{1/2} &\leq C(d,p,q) \| \phi \|_{L^{\infty}} \| g \|_{L^{p}} \| u \|_{L^{q}} \\
  &\leq C(d,p,q) \| \phi \|_{L^{\infty}} \| b \|_{W^{1,p}} \| u \|_{L^{q}}.
 \end{split}
 \end{equation}
 Therefore, from Equations~\eqref{eq:FullFirstTerm}, \eqref{eq:FullThirdTerm} and \eqref{eq:FullSecondTerm}, we conclude that \eqref{eq:CommutatorDorronsoroConclusion} holds.
 \end{proof}

\begin{remark}\label{rmk:SlightVariationOfTheorem}
     Following the proof of Theorem~\ref{thm:CommutatorDorronsoro}, it can be showed that 
    \begin{equation*}
        \left( \int_0^{\infty} \left| \int_{\R^d} ( R^{\eps}(b,u) \ast \varphi_{\eps} ) \phi \, dx \right|^2 \dfrac{\di \eps}{\eps} \right)^{1/2} \leq C(d,p,q) \| b \|_{W^{1,p}} \left( \| u \|_{L^q} \| \phi \|_{L^{\infty}} + \| u \|_{L^{\infty}} \| \phi \|_{L^{q}} \right).
    \end{equation*}
 \end{remark}

\subsection{Propagation of regularity}

We prove a propagation of regularity result as a consequence of Theorem~\ref{thm:CommutatorDorronsoro}.
To measure regularity, we use semi-norms of the form
\begin{equation}
 \llbracket f \rrbracket_{\kappa}^2 = \int_0^{3/4} \| f \ast \varphi_{\eps} - f \|_{L^2}^2 \dfrac{\di \eps}{|\log \eps|^{\kappa} \eps} \quad \kappa \in (0,1).
\end{equation}
For a fixed $\kappa \in (0,1)$, $\llbracket\, \cdot \, \rrbracket_{\kappa}^2$ measures $\log^{1-\kappa}$-regularity.
Note that this semi-norm is weaker than $[ \, \cdot \, ]_{\log}$ introduced in \eqref{eq:log-seminorm}. The reason for introducing $\llbracket\, \cdot \, \rrbracket_{\kappa}$ is that we are not able to propagate $[ \, \cdot \, ]_{\log}$-regularity using Theorem~\ref{thm:CommutatorDorronsoro}.

\begin{lemma}\label{lemma:RootsOfLogRegularity}
 There exists $R = R(\varphi) \geq 2$ such that for any $\kappa \in (0,1)$ and any $f \in \mathcal{S}(\R^d)$
 \begin{equation}\label{eq:RootsOfLogRegularityEquationToProve}
  c(\kappa, \varphi) \int_{\R^d} \chi_{\{ |\xi| > R \}} (\log |\xi|)^{1-\kappa} |\hat{f} (\xi)|^2 \di \xi \leq \llbracket f \rrbracket_{\kappa}^2 \leq C(\kappa, \varphi) \int_{\R^d} (1 + \chi_{\{ |\xi| > 1 \}} (\log |\xi|)^{1 - \kappa}) |\hat{f} (\xi)|^2 \di \xi.
 \end{equation}
\end{lemma}

\begin{proof}
 By Plancherel's identity,
 \begin{equation}\label{eq:RootsOfLogRegularityApplicationPlancherel}
  \llbracket f \rrbracket_{\kappa}^2 = \int_{\R^d} |\hat{f}(\xi)|^2 \int_0^{3/4} |1 - \hat{\varphi}(\eps \xi)|^2 \dfrac{\di \eps}{|\log \eps|^{\kappa} \eps} \di \xi.
 \end{equation}
 Using standard properties of the Fourier transform, there exists an $R \geq 2$ such that
 \begin{equation}
  c(\varphi) \chi_{ \{ |\xi| > \frac{R}{2 \eps} \} } \leq |1 - \hat{\varphi}(\eps \xi)|^2 \leq C(\varphi) \min \{ 1, \eps^2 |\xi|^2 \} \quad \forall \xi \in \R^d, \ \eps \in (0, 3/4).
 \end{equation}
 If $|\xi| \geq R$, integrating with respect to the measure $\frac{d \eps}{|\log \eps|^{\kappa} \eps}$ gives
 \begin{equation}
  c(\varphi) \int_{\frac{R}{2 |\xi|}}^{3/4} \frac{\di \eps}{|\log \eps|^{\kappa} \eps} \leq \int_0^{3/4} |1 - \hat{\varphi}(\eps \xi)|^2 \dfrac{\di \eps}{|\log \eps|^{\kappa} \eps} \leq C(\varphi) \left( |\xi|^2 \int_0^{|\xi|^{-1}} \frac{\eps}{|\log \eps|^{\kappa}} \, \di \eps + \int_{|\xi|^{-1}}^{3/4} \frac{\di \eps}{|\log \eps|^{\kappa} \eps} \right).
 \end{equation}
 Using that $\frac{d}{dt} [- (1 - \kappa)^{-1} |\log t|^{1 - \kappa}] = t^{-1} |\log t|^{- \kappa}$, we find that, up to taking $R$ sufficiently large,
 \begin{equation}
  c(\kappa, \varphi) \chi_{\{ |\xi| > R \}} (\log |\xi|)^{1 - \kappa} \leq \int_0^{3/4} |1 - \hat{\varphi}(\eps \xi)|^2 \dfrac{\di \eps}{|\log \eps|^{\kappa} \eps} \leq C(\kappa, \varphi) + C(\kappa, \varphi) \chi_{ \{ |\xi| > 1 \} } (\log |\xi|)^{1 - \kappa}.
 \end{equation}
 Plugging this into \eqref{eq:RootsOfLogRegularityApplicationPlancherel} gives \eqref{eq:RootsOfLogRegularityEquationToProve}.
\end{proof}

 \begin{proposition}\label{prop:PropagationOfRegularityWithDorronsoro}
  Let $1 < p,q < \infty$ with $\frac{1}{p} + \frac{1}{q} = 1$. If $u_t$ and $b_t$ solve~\eqref{Transport equation} and satisfy
  \[
   b \in L^1([0,T] ; W^{1,p}(\R^d; \R^d)), \quad \divergence b_t = 0, \quad u \in L^{\infty}([0,T] ; L^q \cap L^{\infty}(\R^d))
  \]
  then for all $\kappa \in (1/2, 1)$
  \begin{equation}\label{eq:PropagationOfRegularityWithDorronsoroIneq}
   \llbracket u_t \rrbracket_{\kappa}^2 - \llbracket u_0 \rrbracket_{\kappa}^2 \leq  C(d,p,q,\kappa) \| b \|_{L^1_t W^{1,p}_x} \| u \|_{L^{\infty}_t L^q_x} \| u \|_{L^{\infty}_t L^{\infty}_x}.
  \end{equation}
 \end{proposition}
 
 \begin{proof}
 As in Equation~\eqref{eq:EnergyBalancePlusCommutator}, we have
\begin{align*}
 \| u_t \ast \varphi_{\eps} - u_t \|_{L^2}^2 &\leq \| u_0 \ast \varphi_{\eps} - u_0 \|_{L^2}^2 + 2 \int_0^t \left| \int_{\R^d} (R^{\eps}(b_s , u_s) \ast \varphi_{\eps} ) u_s \di x \right| \di s \\
 &\qquad + 2 \int_0^t \left| \int_{\R^d} R^{\eps}(b_s , u_s) u_s \di x \right| \di s.
\end{align*}
Integrating with respect to the measure $\frac{d \eps}{|\log \eps|^{\kappa} \eps}$, we get
\begin{align*}
 \llbracket u_t \rrbracket_{\kappa}^2 &\leq \llbracket u_0 \rrbracket_{\kappa}^2 + C \int_0^T \int_0^{3/4} \left| \int_{\R^d} (R^{\eps}(b_s , u_s) \ast \varphi_{\eps} ) u_s \di x \right| \dfrac{\di \eps}{|\log \eps|^{\kappa} \eps} \di s \\
 &\qquad + C \int_0^T \int_0^{3/4} \left| \int_{\R^d} R^{\eps}(b_s , u_s) u_s \di x \right| \dfrac{\di \eps}{|\log \eps|^{\kappa} \eps} \di s.
\end{align*}
By H\"older's inequality and Theorem~\ref{thm:CommutatorDorronsoro}, the last addend of the right-hand side is estimated by
\begin{align*}
 &C \int_0^T \left( \int_0^{3/4} \left| \int_{\R^d} R^{\eps}(b_s , u_s) u_s \di x \right|^2 \dfrac{\di \eps}{\eps} \right)^{1/2} \left( \int_0^{3/4} \dfrac{\di \eps}{|\log \eps|^{2 \kappa} \eps} \right)^{1/2} \di s \\
 &\qquad \leq C(d,p,q,\kappa) \| b \|_{L^1_t W^{1,p}_x} \| u \|_{L^{\infty}_t L^q_x} \| u \|_{L^{\infty}_t L^{\infty}_x}.
\end{align*} 
where in the last inequality we used $\kappa > 1/2$.
The second addend is estimated similarly, using Remark~\ref{rmk:SlightVariationOfTheorem}. Due to Lemma~\ref{lemma:RootsOfLogRegularity}, we find \eqref{eq:PropagationOfRegularityWithDorronsoroIneq}.
\end{proof}

\appendix

\section{Velocity fields with gradient in Hardy space $\mathcal{H}^1$}
An extension of Theorem~\ref{thm:LogRegularityViaHeatFlow} to velocity fields with gradient in $\mathcal{H}^1$ reads as:
\begin{theorem}\label{thm:RegularityThmHardy}
 If $u \in L^{\infty}([0,T] \times \R^d)$ and $b \in L^1([0,T] ; W^{1,1}(\R^d; \R^d))$ solve \eqref{Transport equation} and satisfy
 \begin{equation}
  \nabla_{\sym} b \in L^1([0,T] ; \mathcal{H}^1(\R^d; \R^{d \times d}_{\sym})) \quad \text{and} \quad \operatorname{div} b = 0
 \end{equation}
 then
 \begin{equation}\label{eq:IneqToProveRegularityThmHardy}
  [ u_t ]_{\log}^2 - [ u_0 ]_{\log}^2 \leq C(d) \| \nabla_{\sym} b \|_{L^1_t \mathcal{H}^1_x} \| u \|_{L^{\infty}_{t,x}}^2.
 \end{equation}
\end{theorem}

We use $\mathcal{H}^p(\R^d)$ to denote Hardy spaces on $\R^d$ where $1 \leq p < \infty$.
There are multiple equivalent ways to define $\mathcal{H}^p(\R^d)$ but we only provide one:
Let $\varphi \in C^{\infty}_0(\R^d)$ with $\int_{\R^d} \varphi \di x = 1$, define $\varphi_{\delta}(x) \coloneqq \delta^{-d}\varphi(x / \delta)$ and $M_{\varphi} f (x) = \sup_{\delta > 0} |f \ast \varphi_{\delta}(x)|$.
Then $\mathcal{H}^p(\R^d) = \{ f \in L^1_{\loc}(\R^d) : M_{\varphi} f \in L^p(\R^d) \}$ with $\| f \|_{\mathcal{H}^p} \coloneqq \| M_{\varphi} f \|_{L^p}$.
Note that $\mathcal{H}^p(\R^d) = L^p(\R^d)$ with equivalent norms if $p \in (1,\infty)$, while $\mathcal{H}^1(\R^d)\subsetneq L^1(\R^d)$.
The dual of $\mathcal{H}^1(\R^d)$ is $\BMO(\R^d)$ modulo constants. In particular, for any $f \in \mathcal{H}^1(\R^d)$ and $g \in \BMO(\R^d)$, we have
\begin{equation}\label{eq:H1AgainstBMO}
 \left| \langle f , g \rangle \right| \leq \| f \|_{\mathcal{H}^1} \| g \|_{\BMO}.
\end{equation}

The proof of Theorem~\ref{thm:RegularityThmHardy} relies on the following commutator estimate:
\begin{theorem}\label{thm:MainThmHardy}
 Let $b \in W^{1,1}(\R^d; \R^d)$ with $\nabla_{\sym} b \in \mathcal{H}^1(\R^d; \R^{d \times d}_{\sym})$ and $\operatorname{div} b = 0$. Then for any $u, \phi \in \BMO(\R^d) \cap C^1_c(\R^d)$, it holds
 \begin{equation}
  \left| \int_0^1 \int_{\R^d} C^\eps(b,u) \phi \di x \frac{\di \eps}{\eps} \right| \leq C(d) \| \nabla_{\sym} b \|_{\mathcal{H}^1} \| u \|_{\BMO} \| \phi \|_{\BMO}.
 \end{equation}
\end{theorem}

\begin{remark}\label{rmk:MainThmHardyVariation}
 Under the same assumption as Theorem~\ref{thm:MainThmHardy}, following the same strategy, one can show that
   \begin{equation}
  \left| \int_0^1 \int_{\R^d} P_\eps C^\eps(b,u) \phi \di x \frac{\di \eps}{\eps} \right| \leq C(d) \| \nabla_{\sym} b \|_{\mathcal{H}^1} \| u \|_{\BMO} \| \phi \|_{\BMO}.
 \end{equation}
\end{remark}

We recall that $C^{\eps}(b,u)$ denotes the Ambrosio-Trevisan commutator (see Subsection~\ref{subsec:AT}) and $P_s$ denotes the heat semi-group.

\begin{proof}[Proof of Theorem~\ref{thm:RegularityThmHardy}]
As in Theorem~\ref{thm:LogRegularityViaHeatFlow}, without loss of generality, we assume $u_0 \in \mathcal{S}(\R^d)$ and $b_t \in C_c^{\infty}(\R^d; \R^d)$ uniformly in time so that $u_t \in \mathcal{S}(\R^d)$ for all $t \in [0,T]$. Then, by \eqref{eq:EnergyBalancePlusCommutator}
 \begin{align}
    \begin{split}\label{eq:EnergyBalancePlusCommutatorInApp}
	\| {P_\eps u_t-u_t} \|_{L^2}^2&-\| {P_\eps u_0-u_0} \|_{L^2}^2\\
	&= 2 \int_0^t \int_{\R^d} C^{\eps}(b_s,u_s)\cdot (P_{\eps}u_s-u_s)\di x \di s \\
	&= 2 \int_0^t \int_{\R^d} P_{\eps} C^{\eps}(b_s,u_s) u_s\di x \di s - 2 \int_0^t \int_{\R^d} C^{\eps}(b_s,u_s) u_s\di x \di s.
    \end{split}
 \end{align}
Integrating with respect to $\frac{d}{d \eps}$ leads to
\begin{equation}
 [u_t]_{\log}^2 - [u_0]_{\log}^2 = 2 \int_0^1 \int_0^t \int_{\R^d} P_{\eps} C^{\eps}(b_s,u_s) u_s\di x \di s \frac{\di \eps}{\eps} - 2 \int_0^1 \int_0^t \int_{\R^d} C^{\eps}(b_s,u_s) u_s\di x \di s \frac{\di \eps}{\eps}.
\end{equation}
By regularity of $b_t$ and $u_t$, both integrals on the right-hand side are absolutely integrable and therefore by Fubini, the right-hand side equals
\begin{equation}
 2 \int_0^t \int_0^1 \int_{\R^d} P_{\eps} C^{\eps}(b_s,u_s) u_s\di x \frac{\di \eps}{\eps} \di s - 2 \int_0^t \int_0^1 \int_{\R^d} C^{\eps}(b_s,u_s) u_s\di x \frac{\di \eps}{\eps} \di s
\end{equation}
which, due to Theorem~\ref{thm:MainThmHardy} and Remark~\ref{rmk:MainThmHardyVariation}, is estimated by
\begin{align*}
 2 \int_0^T \left| \int_0^1 \int_{\R^d} P_{\eps} C^{\eps}(b_s,u_s) u_s\di x \frac{\di \eps}{\eps} \right| \di s &+ 2 \int_0^T \left| \int_0^1 \int_{\R^d} C^{\eps}(b_s,u_s) u_s\di x \frac{\di \eps}{\eps} \right| \di s \\
 &\leq C(d) \| \nabla_{\sym} b \|_{L^1_t \mathcal{H}^1_x} \| u \|_{L^{\infty}_{t,x}}^2.
\end{align*}
\end{proof}
The remainder of this section is dedicated to proving Theorem~\ref{thm:MainThmHardy}.

\subsection{Proof of Theorem~\ref{thm:MainThmHardy}}
We will rely on the next proposition.

\begin{proposition}\label{prop:BoundOnT}
 Let $u, \phi \in \BMO(\R^d) \cap C^1_c(\R^d)$ and $1 \leq i,j \leq d$ and set
 \begin{equation}
  T(u,\phi)(x) \coloneqq \int_0^1 \int_0^{\eps} \partial_i P_s u (x) \partial_j P_{\eps - s} \phi(x) \di s \, \frac{\di \eps}{\eps}.
 \end{equation}
 Then 
 \begin{equation}
    \| T(u, \phi) \|_{\BMO} \leq C(d) \| u \|_{\BMO} \| \phi \|_{\BMO}. 
 \end{equation}
\end{proposition}

To prove it, we use the following $\BMO$-to-$L^{\infty}$ regularisation property:
\begin{equation}\label{eq:RegularisationPropertyBMO}
 \| \nabla^k P_s f \|_{L^{\infty}} \leq \dfrac{C(d,k)}{s^{k / 2}} \| f \|_{\BMO} \quad \forall k \geq 1.
\end{equation}
and the next lemma which follows from \cite[Section 1]{EFUN75}
\begin{lemma}\label{lemma:BMO-Property-Square-Fct}
 For any $f \in \BMO (\R^d)$, we have 
 \begin{equation}\label{eq:BMO-Property-Square-Fct-Eq}
  \sup_{Q} \int_0^{\ell(Q)^2} \fint_{Q} |\nabla P_s f|^2 \di x \di s \leq C(d) \| f \|_{\BMO}^2
 \end{equation}
 where $\ell(Q)$ denotes the length of the sides of $Q$. 
\end{lemma}

\begin{proof}[Proof of Proposition~\ref{prop:BoundOnT}]
By Fubini (using $u, \phi \in C^1_c(\R^d)$),
\begin{equation}
 T(u, \phi)(x) = \int_0^1 F_t(x) \di t \quad \text{with} \quad F_t(x) \coloneqq \int_0^1 \partial_{i} P_{\eps t} u(x) \partial_j P_{\eps(1-t)} \phi (x) \di \eps.
\end{equation}
We estimate $\| F_t \|_{\BMO}$ for each $t \in (0,1)$.
Let $Q \subseteq \R^d$ be a cube of length $\ell(Q)$ and let $\rho = \ell(Q)^2$ and estimate
\begin{equation*}
 \fint_Q \left| F_t(x) - \fint_Q F_t(y) \di y \right| \di x.
\end{equation*}
Up to exchanging the roles of $u$ and $\phi$, $F_t$ is identical to $F_{1-t}$.
Therefore, it suffices to consider $t \in (0, 1/2]$.
We split $F_t$ in three terms by decomposing the interval of integration in $\eps$ into $(0,a)$, $(a,b)$ and $(b,1)$ with $a = \min (1,\rho)$ and $b = \min (1, \rho / t)$ with the convention that the interval is empty if the two endpoints are identical. We define
\begin{align*}
 F_t^1(x) &\coloneqq \int_0^a \partial_{i} P_{\eps t} u(x) \partial_j P_{\eps(1-t)} \phi (x) \di \eps, \\
 F_t^2(x) &\coloneqq \int_a^b \partial_{i} P_{\eps t} u(x) \partial_j P_{\eps(1-t)} \phi (x) \di \eps, \\
 F_t^3(x) &\coloneqq \int_b^1 \partial_{i} P_{\eps t} u(x) \partial_j P_{\eps(1-t)} \phi (x) \di \eps.
\end{align*}

For $F_t^1$ and $F_t^2$, we will directly estimate the average of their absolute value over $Q$. This is sufficient because
\begin{equation}
    \fint_Q \left| F_t^i - \fint_Q F_t^i \di y \right| \di x \leq 2 \fint_Q |F_t^i| \di x \quad \forall i = 1,2.
\end{equation}
For $F_t^3$, we will estimate its mean oscillation over $Q$.

For the first term, we will only use Lemma~\ref{lemma:BMO-Property-Square-Fct}, for the second one we will use Lemma~\ref{lemma:BMO-Property-Square-Fct} and \eqref{eq:RegularisationPropertyBMO} while for the third one we will use only \eqref{eq:RegularisationPropertyBMO}.
For the first term, using Cauchy-Schwarz,
\begin{equation}
 \fint_Q |F_t^1| \di x \leq \left( \int_0^\rho \fint_Q |\nabla P_{\eps t} u (x)|^2 \di x \di \eps \right)^{1/2} \left( \int_0^\rho \fint_Q |\nabla P_{\eps(1- t)} \phi (x)|^2 \di x \di \eps \right)^{1/2}.
\end{equation}
Applying a change of variable to both factors on the right-hand side and then Lemma~\ref{lemma:BMO-Property-Square-Fct}, we deduce
\begin{equation}\label{eq:FirstTermConclusion}
 \fint_Q |F_t^1| \di x \leq \dfrac{C(d)}{\sqrt{t(1-t)}} \| u \|_{\BMO} \| \phi \|_{\BMO}.
\end{equation} 
For the second term, using Cauchy-Schwarz,
\begin{equation}
  \fint_Q |F_t^2| \di x \leq \left( \int_\rho^{\rho/t} \fint_Q |\nabla P_{\eps t} u (x)|^2 \di x \di \eps \right)^{1/2} \left( \int_\rho^{\rho/t} \fint_Q |\nabla P_{\eps(1- t)} \phi (x)|^2 \di x \di \eps \right)^{1/2}.
\end{equation}
Here we replaced the domain of integration $(a,b)$ by $(\rho, \rho / t)$. If $\rho < 1$, this is justified because $\rho = a < b < \rho / t$. If $\rho \geq 1$, this is justified since $a = 1 = b$.
For the first factor on the right-hand side, we use a change of variable and then Lemma~\ref{lemma:BMO-Property-Square-Fct} to estimate it by $C(d) t^{- 1/2} \| u \|_{\BMO}$. The second factor is estimated using \eqref{eq:RegularisationPropertyBMO}:
\begin{equation}
 \int_\rho^{\rho/t} \fint_Q |\nabla P_{\eps(1- t)} \phi (x)|^2 \di x \di \eps \leq \dfrac{\| \phi \|_{\BMO}^2}{1 - t} \int_\rho^{\rho/t} \frac{1}{\eps} \di \eps \leq \dfrac{\| \phi \|_{\BMO}^2 \log(1/t)}{1 - t}.
\end{equation}
Thus,
\begin{equation}\label{eq:SecondTermConclusion}
 \fint_Q |F_t^2| \di x \leq \dfrac{C(d) \sqrt{\log(1/t)}}{\sqrt{t(1-t)}} \| u \|_{\BMO} \| \phi \|_{\BMO}.
\end{equation}
For the third term, we start by observing that it equals zero if $t \leq \rho$. Therefore, we may assume that $\rho < t$. We compute the gradient of $F_t^3$:
\begin{equation}
 \nabla F_t^3(x) = \int_b^1 \partial_{i} \nabla P_{\eps t} u(x) \partial_j P_{\eps(1-t)} \phi (x) \di \eps + \int_b^1 \partial_{i} P_{\eps t} u(x) \partial_j \nabla P_{\eps(1-t)} \phi (x) \di \eps.
\end{equation}
Using \eqref{eq:RegularisationPropertyBMO} and $t \in (0,1/2]$, the first term is estimated by
\begin{equation}
 C(d) \int_{\rho / t}^1 \frac{ \| u \|_{\BMO} \| \phi \|_{\BMO} }{\eps t \sqrt{\eps(1-t)}} \di \eps \leq C(d) \frac{ \| u \|_{\BMO} \| \phi \|_{\BMO} }{t} \int_{\rho / t}^1 \eps^{- 3 / 2} \di \eps \leq C(d) \frac{ \| u \|_{\BMO} \| \phi \|_{\BMO} }{\sqrt{\rho t} }
\end{equation}
while the second term is estimated by
\begin{equation}
 C(d) \int_{\rho / t}^1 \frac{ \| u \|_{\BMO} \| \phi \|_{\BMO} }{\sqrt{\eps t} \eps(1-t)} \di \eps \leq C(d) \frac{ \| u \|_{\BMO} \| \phi \|_{\BMO} }{\sqrt{t}} \int_{\rho / t}^1 \eps^{- 3 / 2} \di \eps \leq C(d) \frac{ \| u \|_{\BMO} \| \phi \|_{\BMO} }{\sqrt{\rho} }.
\end{equation}
Hence
\begin{equation}\label{eq:BoundGradientThirdTerm}
 \| \nabla F_t^3 \|_{L^{\infty}} \leq C(d) \frac{ \| u \|_{\BMO} \| \phi \|_{\BMO} }{\sqrt{\rho t} }.
\end{equation}
We use \eqref{eq:BoundGradientThirdTerm} to estimate
\begin{equation}\label{eq:ThirdTermConclusion}
 \fint_Q \left| F_t^3(x) - \fint_Q F_t^3(y) \di y \right| \di x \leq \| \nabla F_t^3 \|_{L^{\infty}} \ell(Q) \leq C(d) \frac{ \| u \|_{\BMO} \| \phi \|_{\BMO} }{\sqrt{t} }
\end{equation}
where we used $\sqrt{\rho} = \ell(Q)$.
From \eqref{eq:FirstTermConclusion}, \eqref{eq:SecondTermConclusion} and \eqref{eq:ThirdTermConclusion}, we deduce
\begin{equation}
 \fint_Q \left| F_t(x) - \fint_Q F_t(y) \di y \right| \di x \leq C(d) \left( \dfrac{1}{\sqrt{t(1-t)}} + \dfrac{\sqrt{\log(1/t)}}{\sqrt{t(1-t)}} +  \frac{1}{\sqrt{t} } \right) \| u \|_{\BMO} \| \phi \|_{\BMO} 
\end{equation}
and therefore 
\begin{equation}
 \| F_t \|_{\BMO} \leq \dfrac{C(d) \sqrt{\log(1/t)}}{\sqrt{t(1-t)}} \| u \|_{\BMO} \| \phi \|_{\BMO}.
\end{equation}
From this, we deduce that $T(u,\phi) \in \BMO(\R^d)$ and
\begin{equation}
 \| T(u,\phi) \|_{\BMO} \leq C(d) \| u \|_{\BMO} \| \phi \|_{\BMO} \int_0^{1/2} \dfrac{\sqrt{\log(1/t)}}{\sqrt{t(1-t)}} \di t \leq C(d) \| u \|_{\BMO} \| \phi \|_{\BMO}.
\end{equation}
\end{proof}

\begin{proof}[Proof of Theorem~\ref{thm:MainThmHardy}]
By Proposition~\ref{prop: Ambrosio-trevisan commutator} \ref{item:RewritingOfCommutator},
\begin{equation}
 \int_{\R^d} C^\eps(b,u) \phi \di x = 2 \int_0^\eps \int_{\R^d} \nabla_{\sym} b \nabla P_s u \nabla P_{\eps - s} \phi \, \di x \di s
\end{equation}
 and therefore by Fubini (using $u, \phi \in C^1_c(\R^d)$),
 \begin{equation}
  \int_0^1 \int_{\R^d} C^\eps (b,u) \phi \di x \frac{\di \eps}{\eps} = \sum_{i,j = 1}^d \int_{\R^d} (\nabla_{\sym} b)_{ij} \underbrace{\int_0^1 \int_0^\eps \partial_i P_s u \partial_j P_{\eps - s} \phi \di s \frac{\di \eps}{\eps}}_{=: T_{ij}(u,\phi)} \di x.
 \end{equation}
 By Proposition~\ref{prop:BoundOnT}, $\| T_{ij}(u, \phi) \|_{\BMO} \leq C(d) \| u \|_{\BMO} \| \phi \|_{\BMO}$ and by \eqref{eq:H1AgainstBMO}, we get
 \begin{equation}
  \left| \int_0^1 \int_{\R^d} C^{\eps}(b,u) \phi \di x \frac{\di \eps}{\eps} \right| \leq C(d) \| \nabla_{\sym} b \|_{\mathcal{H}^1} \| u \|_{\BMO} \| \phi \|_{\BMO}.
 \end{equation}
\end{proof}

\bibliographystyle{plain}
\bibliography{biblio}
 
\end{document}